\documentclass[article]{amsart}
\usepackage{amsthm,amsfonts}
\usepackage{amssymb,graphicx}
\usepackage[all]{xy}

\newtheorem{theorem}{Theorem}[section]

\newtheorem{example}[theorem]{Example}

\newtheorem{conjecture}[theorem]{Conjecture}

\newcommand{\C}{\mathbb C}
\newcommand{\N}{\mathcal N}

\newcommand{\R}{\mathbb R}

\newcommand{\0}{{\bf 0}}
\newcommand{\x}{{\bf x}}

\newcommand{\bv}{{\bf v}}

\begin{document}

\title{ Exceptional rays and bilipschitz geometry of real surface singularities }


\author[D. O'Shea]{Donal O'Shea}
\address{President's Office, New College of Florida, 5800 Bayshore Road, Sarasota, FL 34243, USA}
\email{doshea@ncf.edu}

\author[L. Wilson]{Leslie Wilson}
\address{Mathematics Department, University of Hawai`i at Manoa, 2565 McCarthy Mall, Honolulu, HI 96822}



\begin{abstract}
It is known that ambient bilipschitz equivalence preserves tangent cones.
This paper explores the behavior of the Nash cone and, in particular, exceptional
rays under ambient bilipschitz equivalence for real surfaces in $\R^3$ with
isolated singularity.
\end{abstract}

\maketitle

\section{Introduction}
In \cite{OW}, we extended work of Whitney \cite{W}, L{\^ e} \cite{L}, Teissier \cite{LT}, and others \cite{LH} on
limits of tangent spaces in the complex analytic setting to the case of real surfaces in $\R^3$.  In recent years,
there has been much progress on bilipschitz geometry for complex analytic surfaces (see, for example, \cite{BFN3}, \cite{BNP}), and it is again natural to ask whether, and how, results in the complex analytic case carry over to the reals.

To be more precise, and to fix notation, we let $V$ be a semialgebraic surface in $\R^3$ containing the origin ${\bf 0}$ (although all our results are stated in the semialgebraic 
category, they should be true in the subanalytic category as well). Two natural semialgebraic sets, the (Zariski) tangent cone, and the Nash cone, reflect the local geometry of $V$ at $\0$.  The tangent cone, $C \equiv CV \equiv C^+(V, {\bf 0})$, denotes the set of tangent vectors: that is, $\bv \in C$ if and only it there exist $\x_n \in V - \{\0\}, \x_n \rightarrow \0$ and a sequence of positive real numbers $t_n > 0$ such that $t_n\x_n \rightarrow \bv$.  The Nash cone,  $\N\equiv \N V \equiv \N(V, \0)$ denotes the set of 2-planes $T$ with the property that there exists a sequence of $\x_n$ of smooth points of $V$ (by which we mean points where $V$ is locally a 2-dimensional $C^1$ manifold) converging to $\0$ such that $T$ is the limit of tangent spaces to $V$ at the points $\x_n$. By passing to a subsequence if necessary, we can assume that the sequence $\{\x_n\}$ approaches the origin tangent to some ray $\ell$.  Necessarily, $\ell \subset C$.  We let $\N_\ell(V, \0) \subset \N(V,\0)$ denote the space of limits of tangent spaces that can be obtained as limits of tangent spaces along sequences tending to the origin tangent to $\ell$.  Whitney shows that if $V$ is algebraic, then $T\in \N_\ell$ implies $\ell \subset T$ (a result that extends easily to the case $V$ semialgebraic).

In \cite{OW}, we establish the analog of the L{\^ e}-Teissier theorem for algebraic surfaces $V\subset \R^3$ containing the origin $\0$;  that is, for surfaces given implicitly by an equation $\{  f = 0\}$ where $f\in \R[x,y,z]$ is a polynomial vanishing at $\0$.  However, the techniques and results of \cite{OW} apply to semialgebraic surfaces in $\R^3$.  In particular, we show that if  $V \subset \R^3$ is a reduced, semialgebraic surface with $\0$ an isolated singularity, then there exist finitely many rays $\ell_1, \ldots, \ell_r$ in $C$, called exceptional rays, with $\N_{\ell_i}$ connected, closed and one-dimensional.  For any other ray $\ell \in C - \{ \ell_1, \ldots, \ell_r\}$, $\N_\ell(V, \0)$ is a single point (that is a single plane), and   $\N_\ell(V, \0) = \N_\ell(C, \0)$.  An exceptional ray $\ell$ is said to be {\it full} if  $\N_\ell$ consists of the full pencil of planes in $\R^3$ containing $\ell$.   In the case of complex analytic surfaces, all exceptional lines are full, so that knowledge of the tangent cone and exceptional rays completely characterizes the Nash cone, so that L{\^ e}-Teissier's work \cite{LT} together with that of Birbrair, Neumann and Pichon \cite{BNP} allows one to sketch out the basics of a theory of bilipschitz geometry for complex surfaces.  What of real surfaces?

\section{Exceptional rays necessitated by the topology}

A map $h: V \rightarrow W$ between two metric spaces $(V, d_V)$ and $(W, d_W)$ is said to be {\it lipschitz} if $$d_W(h(x), h(y)) \le Kd_V(x, y)$$ for all $x,y \in V$ and some constant $K>0$, and {\it bilipschitz} if $h^{-1}$ exists and is lipschitz.  Equivalently, $h: V\rightarrow W$ is bilipschitz if and only if there exists $K> 0$ such that 
$$\frac{1}{K} d_V(x, y) \le d_W(h(x), h(y)) \le Kd_V(x, y).$$
A semialgebraic set $V$, real or complex, embedded in $\R^n$ or $\C^n$  has two natural metrics. One, the {\it intrinsic} or {\it inner} metric on $V$ is the metric  induced on $V$ by defining the distance $d_i(x,y)$ between two points $x$ and $y$ to be the 
infimum of the lengths of piecewise analytic arcs on $V$ joining $x$ and $y$.  The {\it outer} metric on $V$ defines the distance between any two points $x$ and $y$ to be their Euclidean distance $d_o(x,y)=|x-y|$ in the ambient space.  
Two such sets $V,W$ will be said to be inner (resp. outer) bilipschitz homeomorphic if they 
are bilipschitz homeomorphic with respect to the inner (resp. outer) metrics.  
If we don't say otherwise, we will mean outer.  In addition, if the outer bilipschitz 
homeomorphism is the restriction of a bilipschitz homeomorphism
 on a neighborhood of the sets in Euclidean space, we will say they are ambient 
bilipschitz homeomorphic. A homeomorphism is semialgebraic if its graph is semialgebraic.
All our bilipschitz homeomorphisms are assumed to preserve the origin.

In \cite{S}, Sampaio shows that two semialgebraic sets that are outer bilipschitz homeomorphic have outer bilipschitz homeomorphic tangent cones.  Although it is no longer quite true over the reals that exceptional rays together with the tangent cone completely characterize the Nash cone, the exceptional rays play an important role, and it is natural to ask whether bilipschitz homeomorphic semialgebraic sets have the same exceptional rays up to bilipschitz equivalence.     We shall see shortly that this is not the case.  Nonetheless, there are instances in which (see \cite{OW}) the topology of the tangent cone forces the existence of exceptional rays.  In such cases, two semialgebraic surfaces which are bilipschitz equivalent necessarily have exceptional rays. Three cases are worth singling out. All surfaces 
are understood to be in $\R^3$.

\begin{theorem}\label{thm: first}  
Let $V$ be a real semialgebraic surface with isolated singularity $\0$ and tangent cone $CV$.
 
a.  If $CV$ is a union of rays, each necessarily exceptional and full by \cite{OW},  then any semialgebraic surface  $W$ with isolated singularity $\0$ that is outer bilipschitz homeomorphic to $W$  has tangent cone consisting of the same number of rays, each of which is exceptional and full.  

b.  If $CV$ is bilipschitz homeomorphic to a half-plane, then the two diametrically opposite rays bounding $CV$ are necessarily exceptional and full by \cite{OW}.  Any semialgebraic surface $W$ with isolated singularity $\0$ that is outer bilipschitz homeomorphic to $V$ has two exceptional rays, both full, which together bound $CW$.  (It may also have other exceptional rays).

c.  If $CV$ is bilipschitz homeomorphic to   
three or more half planes meeting along a common axis (that is, a finite pencil with three or more half planes), then any semialgebraic surface $W$ with isolated singularity at $\0$ that is outer bilipschitz homeomorphic to $V$ has an exceptional ray.  

\end{theorem}

\begin{proof}
a.  By Sampaio \cite{S}, $CW$ is bilipschitz homeomorphic to $CV$, hence a union of the same number of rays as $CV$.  By \cite{OW}, each ray is exceptional and full, and hence $\N V$ and $\N W$ are the union of the same number of pencils of planes (however property of these rays being opposite to each other is not necessarily preserved).  Parts b and c follow similarly.  In c, it is worth noting that the topological singular sets of $CV$ and $CW$ each
consist of two rays which must be exceptional, but there may also be other exceptional rays in $CV$ and/or $CW$.
\end{proof}

\section{The case when the tangent cone is a plane}

The most conspicuous case not addressed by Theorem \ref{thm: first} above occurs when a surface has tangent cone bilipschitz homeomorphic to a real plane.  In particular, we consider a surface $V\subset \R^3$ which is the graph $V = \Gamma f$ of a 
semialgebraic function $f: U \subset \R^2\rightarrow \R$, where $U$ is a neighborhood of $0$.  The surface $V$ is then homeomorphic to $U$ by orthogonal projection.
Even if $V$ has an exceptional ray, it may happen that $V$ is bilipschitz homeomorphic to a subset of the $\R^2$. An example is 
$V =\{z^3 = (x-kz)y^6,\, k\ge 0\}$ from \cite{OW}.  Here, $V$ is 
bilipschitz equivalent to its tangent cone $CV =\R^2$, and has exceptional line the $y$=axis with
Nash cone consisting of all planes containing the $y$-axis and slope in the $xz$-plane between 0 and $1/k$.   
 We omit the details for this example, but will instead give the details for a different family of examples below.

\begin{example}\label{ex:one}\end{example}
For any integers $a,b\ge 1$, the real algebraic variety
\begin{eqnarray*}
F(x,y,z) = (x^2+y^{2a})z -y^{2a+b} =0
\end{eqnarray*}
consists of the union of the $z$-axis with the graph of 
\begin{eqnarray*}
f(x,y)&=& \frac{y^{2a+b}}{x^2+y^{2a}} \quad \hbox{when}  (x,y) \ne (0,0)\\
                  f(0,0)&=& \quad 0.
\end{eqnarray*}
$V = \Gamma f$ is a semialgebraic set and, since $\nabla F = (0,0,0)$ only on the $z$-axis,  $V$ is an analytic manifold except at $(0,0,0)$.  Assume $b\ge 2$; then $CV=\R^2$, as the calculations below show.

Consider an analytic arc $A$ in $V$
not tangent to the $y$-axis, lying above the plane arc 
$$ \{ y = |x|^s\cdot\hbox{unit},\quad  x>0 \hbox{ or }x<0, s \hbox{ rational, } s\ge 1\}.$$
For simplicity, we restrict to the case $x>0$. We have
\begin{eqnarray*}
z&=& f(x,y(x))= \frac{x^{(2a+b)s}\cdot\hbox{unit}}{x^2+x^{2as}\cdot\hbox{unit}} = x^{(2a+b)s-2}\cdot\hbox{unit}\\
f_x &=&  \frac{ -2xy^{2a+b}}{(x^2+y^{2a+b})^2} =\frac{x\cdot x^{(2a+b)s}}{x^4}\cdot\hbox{unit} = x^{(2a+b)s-3}\cdot\hbox{unit}\\
&\rightarrow &0\quad \hbox{since} \ (2a+b)s- 3 \ge 1\\ 
f_y &=& {(2a+b)y^{2a+b-1}(x^2+y^{2a}) -2ay^{2a-1}y^{2a+b} } \over{(x^2 + y^{2a})^2}\\
&=&x^{(2a+b-1)s -2} \cdot \hbox{unit} + x^{(2a-1)s + (2a+b)s - 4}\cdot\hbox{unit}\\
&\rightarrow&0\quad \hbox{since}\  (2a+b-1)s- 2\ge 1\ \hbox{and}\  (2a-1)s + (2a+b)s - 4\ge 1.
\end{eqnarray*}
Thus $TV|_A \rightarrow \R^2$ as $x\rightarrow 0$, and $TV|_A \not\rightarrow \R^2$ can only occur for an arc $A$ tangent to the $y$-axis.  That is, the only possible exceptional rays are the positive
and negative parts of the $y$-axis.

Now consider an arc $A$ in $V$ tangent to the $y$-axis (and again assume for simplicity that $x\ge 0$), lying above the plane arc
 $ \{x = y^s\cdot\hbox{unit}\}, s>1, y\ge 0 \ \hbox{or} \ y\le 0.$  Along $A$, we have
\[
z=f(x(y),y) 
= \frac{y^{2a+b}}{ y^{2s}\cdot\hbox{unit} + y^{2a} }
=\begin{cases}
y^{2a+b-2s} \cdot \hbox{unit}\ \hbox{if}\  s\le a\\
y^b\cdot\hbox{unit}\ \quad\qquad  \hbox{if}\ s\ge a
\end{cases}.
\]
Both exponents are greater than $b \ge 2$.  We have
$$
f_y|_A = \frac{(2a+b)y^{2a+b-1}}{y^{2s}\cdot\hbox{unit} +y^{2a}} - \frac{2ay^{(2a+1)+(2a+b)}}{(y^{2s}\cdot\hbox{unit}+y^{2a})^2}.
$$
We set $w = 
\begin{cases} 
s\ \hbox{if} \ s\le a\\
a\  \hbox{if}\ s\ge a
\end{cases}
$
and write  
$y^{2s}\cdot\hbox{unit} + y^{2a} = y^{2w}\cdot\hbox{unit}$.  Hence
$$
f_y|_A = y^{2a+b-1-2w} \cdot\hbox{unit} + y^{4a+b-1-4w}\cdot\hbox{unit}.
$$
Since $w\le a$, both powers of $y$ are greater than of equal to $b-1 \ge 1$, whence
$$
f_y|_A \rightarrow 0 \quad \hbox{as}\quad y \rightarrow 0.
$$
So the limit of $TV|_A$ as $y\rightarrow 0$ is determined by
$$
f_x|_A = y^{2a+b+s-4w}\cdot\hbox{unit}.
$$
For fixed $y$, $f_x$ achieves a max or min only where $f_{xx}= 0$.  Computing, we have
$$
f_{xx} = -2y^{2a+b}\frac{(y^{2a}-3x^2)}{(x^2 + y^{2a})^3} = 0
$$
on  
$$
A: \{ x^2 = \frac{1}{3}y^{2a} \}\quad\hbox{so}\quad x = \pm\frac{1}{\sqrt{3}}y^a.
$$
So $s=a$ and 
$$
f_x|_A = \frac{y^a y^{2a+b}}{y^{4a}}\cdot\hbox{unit} = y^{b-a}\cdot\hbox{unit}.
$$
\bigskip

\noindent{\bf Case I.}  If $b>a$, then $f_x|_A \rightarrow 0$ as $y\rightarrow 0$, so $\N(y$-axis) $= \R^2$ and there is no exceptional ray.  
\bigskip

\noindent{\bf Case II.}  If $b=a$,
$$
f_x|_A = -2(\pm\frac{1}{\sqrt{3}}y^a)y^{2a+b}/(\frac{1}{3}y^{2a}+y^{2a})^2 = \mp\frac{2}{\sqrt{3}}\cdot\frac{9}{16} y^{b+a}.
$$
So $\N(y$-axis) consists of planes containing the $y$-axis with slope in the $x$-direction varying  over the interval [$-\frac{3\sqrt{3}}{8}, \frac{3\sqrt{3}}{8}$].  So we have an exceptional ray that is not a full pencil. 

For $b\ge a$, the tangent spaces to $\Gamma f$ have slopes which are bounded away from infinity, hence $f$ is a lipschitz function and $V = \Gamma f$ is ambient bilipschitz to $\R^2$.
\bigskip

\noindent{\bf Case III.}  If $b<a$, then 
$$
f_x|_A = y^{b-a}\cdot\hbox{unit} \rightarrow \pm\infty \quad\hbox{as}\quad y\rightarrow 0,
$$
so $\N(y\hbox{-axis})$ is a complete pencil. For each fixed $y \ne 0$, the restriction of
$V$ is a curve asymptotic to the $x$-axis with two inflection points occurring at the ends
of the curve in Figure \ref{fig:not1reg_s} (or its reflection about the $x$-axis if $y<0$ and $b$ is odd).
 
\begin{figure}[h] 
  \centering
  \includegraphics[width=3in,height=2.5in,keepaspectratio]{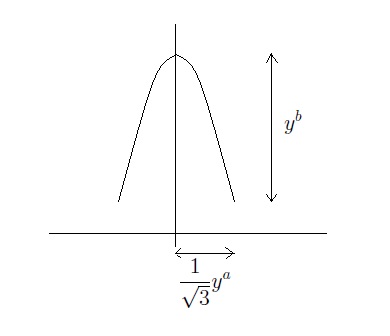}
  \caption{$F = 0$ in planar section $y = \hbox{const}$}
  \label{fig:not1reg_s}
\end{figure}

Denote the two inflection points by
$$
A_\pm(y) =  (\pm\frac{1}{\sqrt{3}}y^a, y, 
\frac{3}{4}y^b)) 
$$
and the maximum by
$B(y) = (0, y, y^b)$.  The Euclidean distance $d_o(A_+, A_-)$ between $A_+$ and $A_-$ is $\frac{2}{\sqrt{3}}y^a$.  The length of the graph connecting them is greater than or equal to 
$$
d_o(A_-,B) + d_o(A_+, B) = 2\sqrt{\frac{1}{3}y^{2a}+y^{2b}} = y^b\cdot\hbox{unit}.
$$
Let 
$d_i(A_-,A_+)$ be the intrinsic distance in $V$.  We will show that this is also $y^b\cdot\hbox{unit}$.  Were this not true, then for every $y$, there are geodesics $\gamma(y)$ from $A_-(y)$ to $A_+(y)$ with
$$
\ell(\gamma(y)) = d_i(A_-(y), A_+(y)) = o(y^b).
$$
There is a function $Y(y)$ so that 
$$
C(y) = (0, Y(s), f(Y(s))) \in \gamma.
$$
This point splits $\gamma$ into $\gamma_1$ from $A_-(y)$ to $C(y)$ and $\gamma_2$ from $C(y)$ to $A_+(y)$.   We have 
$$\ell(\gamma_i(y)) = o(y^b) \quad\hbox{for}\quad i = 1,2$$
and
\begin{eqnarray*}
\ell(\gamma_i(y)) &\ge& d_o(A_\pm(y), C(y))\\
&=&((y-Y(y))^2 + (A_\pm(y))^2 + f(Y(y))^2)^\frac{1}{2}
\end{eqnarray*}
So
\begin{eqnarray*}
&1)& |y-Y(y)| = o(y^b)\\
&2)& |A_\pm(y)| = y^a = o(y^b)\\
&3)& |f(Y(y))| = o(y^b)
\end{eqnarray*}
The second relation 2) holds because $b<a$.  By 1), $Y(y) = y + o(y^b)$, so that $f(Y(y)) = Y(y)^b\cdot\hbox{unit} = y^b\cdot\hbox{unit}$ violating $3)$.  Therefore,
$$
d_i(A_-(y), A_+(y)) = y^b\cdot\hbox{unit}\ >> \ d_o(A_-(y), A_+(y))
$$
so that $V$ is not $\ell$-regular, which implies $V$ cannot be outer bilipschitz to $\R^2$ (we discuss $\ell$-regularity in section 4).

One can ask whether   bilipschitz equivalence to the plane is only possible when the exceptional rays are not full.  The following example  shows that (even semialgebraic) ambient bilipschitz maps need not preserve full exceptional rays. 

\begin{example}\label{ex:two}\end{example}
Let $V$ be the surface
\begin{eqnarray*}
z^3 = f(x,y)^3 &=& (x^2 + y^2)(x^2-y^3) \\
&=& x^4 +x^2(y^2-y^3) - y^5.
\end{eqnarray*}
$V$ is semialgebraic, analytically nonsingular except at $\0$ and is the graph of $f$.  Calculations similar to those in Example \ref{ex:one} establish that 
$CV = \R^2$ 
and that the only possible exceptional ray is the positive $y$-axis.  

Consider the family of plane arcs $\{ \gamma_C \} $ with $C$ a constant 
$$\gamma_C = \{x = y^{\frac{3}{2}} + Cy^{\frac{7}{4}} \}.$$
On $\gamma_C$, 
$$
z^3 = y^6 + \ldots + (y^3 + 2Cy^{\frac{13}{4}} + \ldots)(y^2 - y^3) - y^5,
$$
where the dots $\ldots$ indicate higher order terms.  So
\begin{eqnarray*}
z^3 &=& 2Cy^\frac{21}{4} + \ldots \\
z &=& (2C)^\frac{1}{3}y^\frac{7}{4} + \dots\,.
\end{eqnarray*}
Let
$$
F(x,y,z) = x^4 + x^2(y^2-y^3) - y^5 - z^3.
$$
On an arc $\gamma_C$, we have 
\begin{eqnarray*}
F_x&=& 4x^3 + 2x(y^2-y^3)\\
&=& 4(y^\frac{9}{2} + \ldots) + 2(y^\frac{3}{2} + \ldots)(y^2 - y^3)\\
&=& 2y^\frac{7}{2} + \ldots ,\\
F_y&=& x^2(2y - 3y^2) - 5y^4\\
&=& (y^3 +\ldots)(2y-3y^2) - 5y^4\\
&=&-3y^4 + \ldots,\\
F_z&=& -3z^2 = -3(2C)^\frac{2}{3}y^\frac{7}{2} + \ldots\, .
\end{eqnarray*}
So, $\nabla F \rightarrow (2,0, -3(2C)^\frac{2}{3})$ as $y\rightarrow 0$ along the arc. 

Similarly, as $y\rightarrow 0$ along $x = -y^{\frac{3}{2}} + Cy^{\frac{7}{4}}$, we have 
$\nabla F \rightarrow (2,0, 3(2C)^\frac{2}{3})$.  We conclude that the limits of tangent planes to $V$ along these arcs consist of all planes containing the $y$-axis:  for each $C\ne 0$, the tangent plane intersected with the $xz$-plane has slope $\pm\frac{2}{(3(2C)^\frac{2}{3})}$ (so all nonzero numbers) and when $C=0$ the slope is infinite.  Thus, the positive $y$-axis is a full exceptional ray.  

We now show that $V$ is semialgebraic bilipschitz equivalent to $\R^2$.  Consider the plane arcs
\begin{eqnarray*}
\alpha_\pm:& x =\pm(y^\frac{3}{2} + y^\frac{7}{4})\\
\beta_\pm:& x =\pm(y^\frac{3}{2} - y^\frac{7}{4})
\end{eqnarray*}
Let $D_1, D_2, D_3$ be the regions bounded between the positive $y$-axis and $\beta_+$, between
$\beta_+$ and $\alpha_+$, and between $\alpha_+$ and $x=y,\, y\ge 0$, respectively.  
Then $f$ is lipschitz on $D_1$ and $D_3$ and $f_x > 0$ on $D_2$ 
(and $f_x\rightarrow \infty$ on $x=y^{3/2}$).  For each fixed $y$ the change in $x$ 
on $D_1$ (resp. $D_3$) is $\Delta x=y^{3/2}\cdot$unit  (resp.
$y^1\cdot $unit), while both the change in $x$ and the 
change in $z$ on $D_2$ is $\Delta z=y^{7/4}\cdot $unit. So the $\Delta z$ on $\Gamma f|D_2$ 
goes to zero faster than the $\Delta x$ on $D_1$ and $D_3$, which is the hypothesis for  Theorem 4.2 a).  
We can define $D_1^-, D_2^-, D_3^-$ similarly using $\beta_-, \alpha_-,$ and $\{ x=-y,\, y\ge 0\}$, which shows by Theorem 4.2 that all of $V$ is semialgebraic bilipschitz homeomorphic to $\R^2$.  

\begin{example}\label{ex:three}\end{example}  The following example is similar, but involves more arcs of inflection points with infinite slope tangent to the same exceptional ray.  Let
\begin{eqnarray*}
z^3&=& (x^2-y^5)(x^2-y^7)\\
&=& x^4-x^2(y^5 + y^7)+y^{12}.
\end{eqnarray*}
The tangent planes of infinite slope occur where the right hand side is equal to 0, so that $x=\pm y^\frac{5}{2}$ or $x=\pm y^\frac{7}{2}$.  Consider
$$
x= y^\frac{7}{2} + Cy^\frac{17}{4} 
\ \hbox{where} \ C\ \hbox{is a constant}.
$$
Along these arcs, 
\begin{eqnarray*}
z^3&=&(y^7 +\ldots -y^5)(\not{y^7} + 2Cy^\frac{31}{4} +\ldots -\not{y^7})\\
&=& -(2C)y^\frac{51}{4} +\ldots\ \hbox{and} \\
z^2 &=& (2C)^\frac{2}{3} y^\frac{17}{2} +\ldots.\\
\hbox{So}\ \ \ \  \ 3z^2z_x &=& 4x^3- 2x(y^5 +y^7)\\
&=& 4(y^\frac{21}{2} +\dots) -2(y^\frac{7}{2} +\ldots )(y^5 +\ldots)\\
&=& -2y^\frac{17}{2} +\ldots\\
z_x&=& \frac{-2}{3(2C)^\frac{2}{3}} +\ldots \rightarrow \frac{-2}{3(2C)^\frac{2}{3}} \hbox{as} \ y\rightarrow 0.
\end{eqnarray*}
Similarly, one obtains $z_y = y^\frac{5}{2}\cdot\hbox{unit}\rightarrow 0$ as $y\rightarrow 0$ so that the tangent planes to $V$ along these arcs approach all planes containing the $y$-axis with negative slope.

Now let $x = y^\frac{5}{2} + Cy^\frac{15}{4} $ with $C$ a constant.  We have
\begin{eqnarray*}
z^3 &=& (\not{y^5} + 2Cy^\frac{25}{4} +\ldots - \not{y^5})(y^5 + \ldots - y^7)\\
&=& 2Cy^\frac{45}{4} + \ldots
\end{eqnarray*}
and
\begin{eqnarray*}
z^2 &=& (2C)^\frac{2}{3}y^\frac{15}{2} + \ldots...\\
3z^2z_x &=& 4(y^\frac{15}{2} +\ldots) - 2(y^\frac{5}{2} +\dots)(y^5+\ldots)\\
&=& 2y^\frac{15}{2} + \ldots .
\end{eqnarray*}
We have
$$
z_x = \frac{2}{(2C)^\frac{2}{3}} + \ldots \rightarrow \frac{2}{(2C)\frac{2}{3}} \quad \hbox{as} \quad y \rightarrow 0.
$$
Moreover, $z_y = y^\frac{3}{2} \cdot\hbox{unit} \rightarrow 0$ as $y\rightarrow 0$ so the tangents to $V$ along these arcs approach all the planes containing the $y$-axis with positive slope.

Consider the regions 
$$D_1, D_2, D_3, D_4, D_5, D_6$$ 
bounded by the arcs 
$$
x=0, y^\frac{7}{2} - y^\frac{17}{4}, y^\frac{7}{2}+y^\frac{17}{4}, y^3, y^\frac{5}{2} - y^\frac{15}{4}, y^\frac{5}{2} +y^\frac{15}{4}, 2y^\frac{5}{2}.
$$
Note $f$ is lipschitz on $D_1, D_3, D_4, D_6$  and $f_y$ has constant sign on $D_2, D_5$.  Now the  $D_1, D_3$ are wider ($\Delta x=y^{7/2}\cdot $unit and $\Delta x=y^{3}\cdot $unit) than the $z$ change over
$D_2$ ($\Delta z=y^{17/4}\cdot $unit) and $D_4, D_6$ are wider ($y^{5/2}\cdot $unit and $y^{5/2}\cdot $unit) than the $z$ change over $D_5$ ($y^{15/4}\cdot $unit).  The same holds on the corresponding regions with $x<0$. Hence Theorem 4.2 shows that $V$ is semialgebraic bilipschitz homeomorphic to the plane. 

\section{Length regularity and the main theorems}

A surface $V$ is said to be {\it normally embedded} (see \cite{BM}) if 
its outer and inner metrics are equivalent, i.e. there is a constant $K>0$
such that $d_i(x,y)\le K d_o(x,y)$ for all $x,y\in V$. 
In such a case, we will say that $V$ is {\it length regular} or {\it $\ell$-regular}.  (In \cite{FW}, $\ell$-regularity is referred to as 1-regularity, following \cite{T} (p. 79) where a hierarchy of regularity is defined.) Since outer bilipschitz equivalence 
implies inner bilipschitz equivalence, it is not hard to see that $\ell$-regularity is invariant under outer bilipschitz equivalence.  
In this section we assume that $V\subset \R^3$ is a semialgebraic surface that is a $C^1$-manifold everywhere except $\0\in V$.  Assume further that, for $U$ a semialgebraic neighborhood of $\0$, $V = \Gamma f$ is a graph of semialgebraic function $f: U \rightarrow \R$, continuous at $\0$, $f(0)=0$, and that the tangent cone $C\equiv CV$ is the $xy$-plane.
 
We know that there exist at most finitely many exceptional rays.
If there are no exceptional rays, then, shrinking $U$ if necessary to remove points where 
the tangent to $V$ is vertical, $f$ is lipschitz and $V=\Gamma f$ is bilipschitz to $U$ by
$H(x,y, z)= (x, y, z-f(x,y))$.

In this section we will only discuss local results at $\0$, so will shrink $V$ and $U$ and 
domains of our bilipschitz maps to
smaller neighborhoods of $\0$ as necessary.  The assumption of $\ell$-regularity will
mean that it holds on a small enough neighborhood.  Nevertheless, the proofs  of our 
theorems will show that the results hold globally in certain cases, such as in the 
examples of the last section.

We look at what happens around a single exceptional ray which we take, without loss of generality, to be the positive $y$-axis.
There are rays $x=\pm my$ and an $\epsilon>0$ such that the positive $y$-axis is the only exceptional ray in the wedge $W\subset C$: $0\le y\le \epsilon$, $|x|\le my$.

Consider two analytic arcs $\{ x=r_1(y) \}\ \hbox{and}\ \{ x=r_2(y)\},\ 0\le y\le \epsilon$ in $C$ with
 $ -my<r_1(y)<r_2(y)<my$ for $y>0$.  Let $D=D(r_1,r_2,\epsilon)\subset C$ be the sector bounded by
 these arcs:  $0\le y\le \epsilon, \quad r_1(y)\le x\le r_2(y)$.  Say that a region $\Gamma f|_D$ over $D$ is a
{\it piece} $P$ of $V$ if $P$ is semialgebraic.
There is a positive rational number $w$ called the {\it 
width} of $P$ given by $|r_2(y)-r_1(y)|=y^{1/w}\cdot \hbox{unit}$. 

Both $\max(P)=\{(x,y,z)\in P: z\ge v \text{ for all }(u,y,v)\in P\}$ and the similarly defined $\min(P)$ are semialgebraic,
and there are arcs $(x_{max}(y),y)$ and $(x_{min}(y),y)$, 
such that $(x_{max}(y),y,f(x_{max}(y),y))\in \max(P)$ and $(x_{min}(y),y,f(x_{min}(y),y))\in \min(P)$  for all 
$0\le y \le \epsilon$ (shrinking $\epsilon$ if necessary).  
There is a positive rational number $h$ called the {\it 
height} of $P$ given by $|f(x_{max}(y),y)-f(x_{min}(y),y)|=y^{1/h}\cdot \hbox{unit}$. 

We say that $P$ is {\it flat} ($FL$ for short) if there exists a $K>0$ such that
$|f_x|\le K$ on $D$ (so $f$ is lipschitz on $D$). 

Call $P$ {\it fast increasing} (respectively, {\it fast decreasing}), $FI$ (resp., $FD$) for short, if there exists a constant $K>0$ such that 
$f_x\ge K$ (resp. $f_x \le -K$) on $D$ and there exists $r_3(y)$ such that 
$r_1(y)\le r_3(y)\le r_2(y)$ and $f_x(r_3(y),y)=\infty$
for all $0<y<\epsilon$ or $f_x(r_3(y),y)\rightarrow\infty$ as $y\rightarrow 0$.
Note that if no such $r_3$ exists, then $P$ is flat. 

 Call $P$ {\it never fast increasing} or $NFI$ (resp. $NFD$ ) if $P$ is a union of neighboring pieces (i.e. with common intersection an $arc$) which are either $FL$
or $FD$ (resp. $FL$ or $FI$). 

Let $\pi$ be the orthogonal projection to the $xy$-plane
The piece $\Gamma f|_W$ over the wedge $W$ can be partitioned into consecutive pieces $P_1$, $P_2, 
\ldots , P_{2n+1}$ (consecutive here means that for $0\le y \le \epsilon$ the
$\pi(P_i)$ are bounded by arcs $x=r_i(y)$ and $x=r_{i+1}(y)$, with $r_i(y)$
strictly increasing as $i$ increases, $r_1(y)=-my$, $r_{2n+1}(y)=my$, 
all odd labeled $P_k$ are $FL$ and the even labeled $P_k$ are alternately 
$NFI$ and $NFD$.  $W$ is called {\it well-separated} if there is some 
such partition so that each flat piece has width greater than, or equal to, the heights of its adjacent pieces.


Suppose the exceptional ray (the positive $y$-axis) is not full. Each plane in the Nash 
fiber is determined by its slope in the $x$-direction.  If the Nash fiber lacks the plane with slope infinity, then $f$ is Lipschitz on $W$.  If the Nash fiber includes
the plane with slope infinity but misses a plane of some positive slope, then $W$ has no piece which is FI, so the wedge has FL pieces on the outside of width 1, and an NFI middle piece, which has height less than 1 (by tangency of $V$ to the plane).  So $W$ is well-separated.  The same conclusion holds if the Nash fiber misses some plane of negative slope.  So every non-full exceptional ray lies in a well-separated wedge.

\begin{theorem}{\bf Necessary Condition.}\label{thm: main-nec}

Suppose there exists a piece $P=P_1\cup P_2\cup P_3 \subset \Gamma f|_W$ with $P_1, P_2, P_3$ consecutive such that $P_2$ is $FL$, one of $P_1$ and $P_3$ is $NFI$ and the other is $NFD$, and $w(P)< \min(h(P_1),h(P_3))$.
Then $P$ is not $\ell$-regular, so $V$ is not bilipschitz to the plane.
\end{theorem}

\begin{proof}
Let $P=P_1\cup P_2\cup P_3$ be as above, and assume without loss
of generality that   $P_1$ is $NFD$ and $P_3$ is $NFI$.
Then $P_1$ is a union of pieces which are $FI$ and $FL$.
Since $h(P_1)>w(P)\ge w(P_1)$ and %
the height of every $FL$ piece 
is less or equal to its width, there is an $FI$ piece $P^*$ in $P_1$
with $h_1=h(P^*)=h(P_1)$. 

Consider the arcs $(x^*_{min}(y),y)$ and $(x^*_{max}(y),y)$ associated 
to $P^*$. Since all the slopes of $P^*$ in the $x$-direction are positive,
$x^*_{min}(y)<x^*_{max}(y)$ and $h(P^*)=h(P_1)=h_1$ satisfies 
$|f(x^*_{min}(y),y)-f(x^*_{max}(y),y)|=y^{1/h_1}\cdot \hbox{unit}$.

Similarly $P_3$ contains an $FD$ piece $P^!$ for which 
$x^!_{min}(y)>x^!_{max}(y)$ and $h(P^!)=h(P_3)=h_3$ satisfies 
$|f(x^!_{min}(y),y)-f(x^!_{max}(y),y)|=y^{1/h_3}\cdot \hbox{unit}$.
We may assume without loss of generality  that $f(x^*_{min}(y),y)\ge f(x^!_{min}(y),y)$.
Let $x_1(y)=x^*_{min}(y)$ and
let $x_2(y)=x_{max}(y)$ for $P$. 
Then $x_1(y)<x_2(y)$, and there exists 
$x_3(y)$ with $(x_3(y),y)\in P_3$ and 
$f(x_3(y),y)=f(x_1(y),y)$.  Necessarily $x_2(y)<x_3(y)$.
Now, 
$x_3(y)-x_1(y)$ goes to zero faster than $f(x_2(y),y)-f(x_1(y),y)$,
so the same argument as in Example \ref{ex:one}, Case III, shows that $P$ and hence
$V$ is not $\ell$-regular.

The second part follows from the invariance of $\ell$-regularity under
bilipschitz equivalence.
\end{proof}

\begin{theorem}{\bf Sufficient Condition.}\label{thm:main-suff}

a)  If $\Gamma f|_W$ is well-separated, then it is semialgebraically bilipschitz equivalent to its linearization, which is the graph of a 
lipschitz function over the plane. In particular, this holds if the exceptional ray is 
non-full.

b) If all the exceptional rays lie in well-separated pieces, then $V$ is 
semialgebraically bilipschitz equivalent to $C$.
\end{theorem}

%
 \begin{conjecture} \label{thm:main-conj}

Let $V\subset \R^3$ be semialgebraic surface that is a $C^1$-manifold everywhere except $\0\in V$.  Assume further that $V = \Gamma f$ is a graph of a semialgebraic function $f: U \rightarrow \R$, continuous at $\0$, $U$ a semialgebraic neighborhood of $\0$, and that the tangent cone $C\equiv CV$ is the $xy$-plane.  Then $V$ is bilipschitz to $C$ if and only if $V$ is 
$\ell$-regular.    
\end{conjecture}

\begin{proof} (of Theorem \ref{thm:main-suff})
Let $\alpha(y)\subset \R^2$ and $\beta(y)\subset \R^2$ be arcs on $y\ge 0$ with convergent Puiseux expansions at $0$ of the form
$$
\alpha(y) = (y^au(y), y), \ \beta(y) = (y^bv(y), y); \quad a,b \ge 1; \quad u(y), v(y)\  \hbox{units}
$$
where $u(y)$ being a unit means that it has the form $u_0 + u_1y^{r_1} + \ldots$ with $u_0\ne 0$ and $r_i>0$ rational (and similarly $v(y)$).  We allow $a,b = 1$, and assume that $\pi_1\alpha(y) \le \pi_1\beta(y), \hbox{for all}\  y \ge 0$ (where $\pi_1$ is projection onto the first coordinate). 

\begin{figure}[h] 
  \centering
 \includegraphics[width=3in,height=2.5in,keepaspectratio]{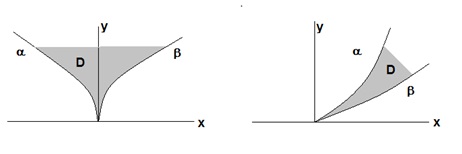}
  \caption{The region $D$}
  \label{fig:alphabeta}
\end{figure}

Let $D$ be the plane region bounded by $\alpha(y)$ and $\beta(y)$. See Figure \ref{fig:alphabeta}.  Although $V$ is $C^1$ off $0$, the function $f$ is not necessarily $C^1$ (the tangent to $V$ could contain vertical vectors).  

We will show how to construct semialgebraic bilipschitz transformations (with compact support) of the graph of $f|_D$ to it's linearization in several situations.
\bigskip

{\bf Case I}.  Assume that $f$ is lipschitz on $D$.  

Let $L$ be the linearization of the $f$ in the $x$-direction:
$$
L(x,y) = f(\alpha(y)) + \frac{f(\beta(y))-f(\alpha(y))}{\pi_1(\beta(y)) - \pi_1(\alpha(y))}\cdot(x - \pi_1(\alpha(y)).
$$
We will construct a semialgebraic bilipschitz map $H(x,y, z)$ such that 
$$
H(x,y, L((x,y)) = (x,y, f(x,y)).
$$
Fix any $c>0$.  Let 
\begin{eqnarray*}
f^T &=& \max(f, L) + c\cdot(|f - L|)\\
f^B &=& \min(f, L) - c\cdot(|f - L|)
\end{eqnarray*}
If $f > L$, then 
\begin{eqnarray*}
f^T &=& f + c(f-L) = (1+c)L - cL\quad \hbox{and}\\
f^B &=& L - c(f-L) = (1+c)L - cf
\end{eqnarray*}
If $f<L$, then
\begin{eqnarray*}
f^T &=& L + c(L-f) = (1+c)L - cf\quad \hbox{and}\\
f^B &=& (1+c)f -cL
\end{eqnarray*}
If $f=L$, then 
$$
f^T = f^B = f = L.
$$
So, for example, for a fixed $y$ positive and $c = \frac{1}{2}$ and $f$ as shown, $f^T$ and $f^B$ are as sketched in Figure \ref{fig:fthing}.
\begin{figure}[h] 
  \centering
  \includegraphics[width=3in,height=2.5in,keepaspectratio]{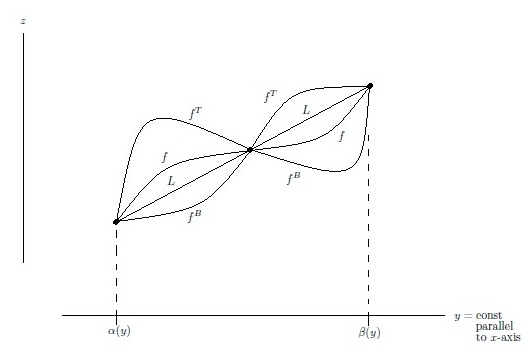}
  \caption{$f^T$ and $f^B$}
  \label{fig:fthing}
\end{figure}
Let  $H(x, y, z) = (x, y, h(x,y, z))$ where $h(x,y,  z)$ maps [$f^B, L$]  linearly onto [$f^B, f$] and [$L, f^T$] linearly onto [$f, f^T$] (for each fixed $x, y$) and $h(x,y,z) = z$ if $z<f^B$ or $z>f^T$.  (See Figure \ref{fig:hgraph}.) $H$ is assumed the identity outside the region bounded between $f^B$ and $f^T$.
\begin{figure}[h] 
  \centering
  \includegraphics[width=3in,height=2.5in,keepaspectratio]{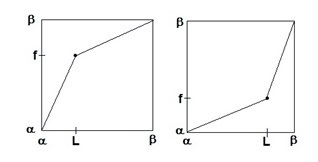}
  \caption{Graph of $h$}
  \label{fig:hgraph}
\end{figure}

It is easy to compute that
\[
h_z = 
 \begin{array}{l|cc}
     &\small f>L&f<L\normalsize \\
     \hline
     & & \\
\small z<L \normalsize &\frac{c+1}{c}&\frac{c}{c+1}\\
&&\\
\small z>L\normalsize&\frac{c}{c+1}&\frac{c+1}{c}
\end{array}
\quad graph\ of\ h\  
\]

To compute $h_x$ and $h_y$ it is convenient to express $h$ in terms of the following ``bump function":
\[
h(x,y,z) = 
\begin{cases}
\frac{z-f^B}{L-f^B} \quad\hbox{on}\qquad\quad f^B \le z \le L\\
\frac{f^T-z}{f^T-L   }\quad\hbox{on}\qquad\quad L \le z\le f^T\\
\ \ \ 0\ \ \ \quad\hbox{outside}\quad f^B \le z \le f^T.
\end{cases}
\]
(See Figure \ref{fig:bumpu}.) Then $h(x,y,z) = z + u(x,y,z)(f(x,y) - L(x,y))$.

\begin{figure}[h] 
  \centering
  \includegraphics[width=3.25in,height=2.75in,keepaspectratio]{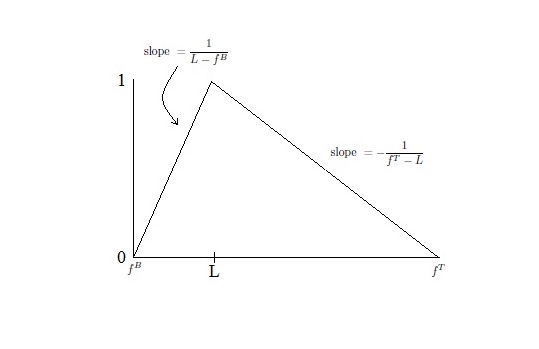}
  \caption{The bump function $u$}
  \label{fig:bumpu}
\end{figure}

By assumption $f$ is lipschitz on $D$; that is, there exists $k > 0$ such that 
$$
|f(x,y) - f(\overline{x}, \overline{y})| \le k\|(x,y)-(\overline{x}, \overline{y})\|\quad\hbox{for all}\   (x,y), (\overline{x}, \overline{y})\ \hbox{in} \  D.
$$
Letting
$$
w(x,y) = \frac{x-\alpha(y)}{\beta(y)-\alpha(y)} \quad 0\le w \le 1,
$$
we can rewrite $L(x,y)$ as 
$$
(1-w)f(\alpha(y)) + wf(\beta(y)).
$$
So
\begin{eqnarray*}
|L(x,y) - L(\overline{x}, \overline{y})| &\le& |f(\alpha(y)) - f(\alpha(\overline{y})|+|f(\beta(y))-f(\beta(\overline{y}))|\\
&\le& k(\|\alpha(y) - \alpha(\overline{y})\| + \|\beta(y)-\beta(\overline{y})\|)\\
&\le&k'\|y - \overline{y}\|
\end{eqnarray*}
where the last inequality follows because $\alpha$ and $\beta$ are lipschitz.  
It follows that $L$ is lipschitz on $D$.

Note that lipschitz implies that the partials are bounded where defined.  We have
$$
h_x = u\cdot(f_x-L_x) + u_x(f-L)
$$
where $0\le u \le 1$ and $(f_x-L_x)$ is bounded.   On $f^B\le z \le L$, the second term
\begin{eqnarray*}
u_x(f-L) &=& \frac{(L-f^B)(-f^B_x) - (z-f^B)(L_x - f^B_x)}{(1-f^B)^2}\cdot(f-L)\\
&=& -f_x^B\cdot\frac{(f-L)}{(L-f^B)} - \frac{(z-f^B)}{(L-f^B)} \cdot(L_x -f^B_x)\cdot\frac{(f-L)}{(L-f^B)}
\end{eqnarray*}
is bounded, because all terms in the second equality are bounded---note 
that $0\le \frac{z-f^B}{L-f^B} \le 1$ since $f^B \le z\le L$. 
Similarly, $f_x$ is bounded on $L\le z \le f^T$.

The calculation of $h_y$ is the same except that we need to replace $f^B_x$ and $L_x$ by $f_y^B$ and $L_y$.

Since $f$ and $L$ are Lipschitz in $x, y$ so are $f^B$ and $f^T$ and the same calculation shows that $h_y$ is bounded.  

The map $H$ is a homeomorphism, is semialgebraic, and there is a semialgebraic nowhere dense set $\Sigma$ off which $H$ is differentiable with
\[
dH = \left(\begin{array}{ccc}
1&0&0\\
0&1&0\\
h_x&h_y&h_z
\end{array}\right)
\quad\hbox{and}\quad 
\begin{cases}
h_x, h_y \ \hbox{bounded}\\
\frac{c}{c+1} \le h_z \le \frac{c+1}{c}
\end{cases}.
\]
We claim that this implies that $H$ is globally bilipschitz.

\begin{proof} To see this, pick $p, q$ in $\R^3$.  The line segment $\overline{pq}$ may intersect $\Sigma$ in infinitely many points (an interval).  Let $A$ be a small disk in the plane perpendicular to $\overline{pq}$ centered at the midpoint of $\overline{pq}$. Then $\{ \sigma: \overline{p\sigma}\cap\Sigma$ is not finite$\} \cup \{\sigma : \overline{\sigma q}\cap\Sigma$ is not finite$\}$ is nowhere dense in $A$, so fix $\sigma$ not in this set.  
Then 
$$
H(q) - H(p) = \int_{p\sigma q} dH\cdot\overline{u}
$$
 where $\overline{u}$ is the unit tangent to $\overline{p\sigma}$ or $\overline{\sigma q}$ as appropriate.  So 
 $\hbox{length}(p\sigma q)\le 2|p-q|$.  Thus 
 $$\| H(q) - H(p)\| \le \max\| dH\| \cdot 2|q-p|$$
 where $\max\| dH\|$ is a constant $K$.  So $H$ is lipschitz.  
 Now 
 \[
dH^{-1} = \left(\begin{array}{ccc}
1&0&0\\
0&1&0\\
-\frac{h_x}{h_z}&-\frac{h_y}{h_z}&\frac{1}{h_z}
\end{array}\right)
\quad\hbox{off}\quad h(\Sigma).
\]
So $H^{-1}$ is also lipschitz.  Note by the construction that $H$ and $H^{-1}$ are also semialgebraic.
\end{proof}
\bigskip

{\bf Case II}.  Suppose that $f$ is increasing with respect to $x$ on [$\alpha(y), \beta(y)$] for all $y>0$, but is not necessarily lipschitz.  Let $\ell$ be a line through $0$ of positive slope in the $xz$-plane, and $\ell^\perp$ its orthogonal complement in the $xz$-plane. Let $\pi_\ell$ be orthogonal projection to the $\ell\times (y\hbox{-axis}) $ plane (so $\ker \pi_\ell = \ell^\perp$).

Let $V_D$ be the part of the graph of $f$ over $D$.  Because the slope of $f$ in the $x$-direction lies in [$0,\infty$] over $D$, lines parallel to $\ell^\perp$ are not contained in $TV_D$ or its Nash limit at $0$.  Thus $\pi_\ell|(V_D-0)$ is nonsingular with image $\tilde{D}$ in  $\ell \times (y\hbox{-axis}) = \hbox{image of } \pi_\ell$.  $V_D$ is the graph of a lipschitz function
$$
\tilde{f}: \tilde{D} \subset \hbox{im}(\pi_\ell) \rightarrow \ell^\perp.
$$
A rotation about the $y$-axis takes $\hbox{im}(\pi_\ell)$ onto the $xy$-plane and $\ell^\perp$ onto the $z$-axis, so takes $V_D$ to the graph of a lipschitz function. 

By the argument in Case I, we have a bilipschitz map $H$ taking $\Gamma L$ to $\Gamma f$.  This $H$ moves points in the direction of $\ell^\perp$.

The support of $H$ is the region between $\Gamma f^B$ and $\Gamma f^T$ as before.  See Figure \ref{fig:slantproj_s}.

\begin{figure}[h] 
  \centering
  \includegraphics[width=4.5in,height=4in,keepaspectratio]{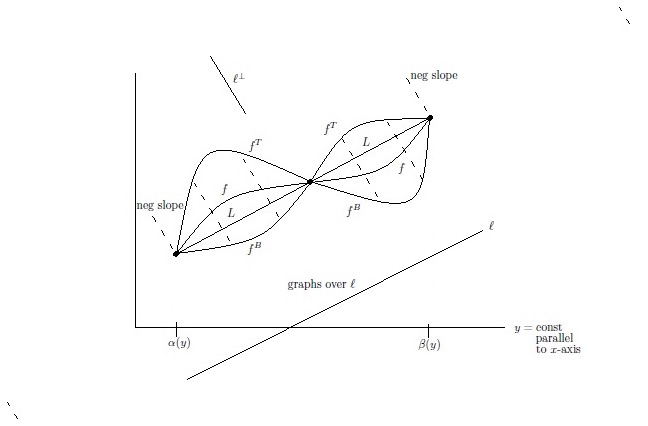}
  \caption{The map H}
  \label{fig:slantproj_s}
\end{figure}
This support can be contained between [$\alpha(y), \beta(y)$] unless $f_x=\infty$ at $\alpha (y)$
and/or $\beta(y)$ (as in Figure \ref{fig:slantproj_s}).  In this latter case, by choosing $c$ (used to define $f^T$ and $f^B$) sufficiently small, we can make the tangent to $f^T$ (or $f^B$) arbitrarily close to vertical in the case $f_x = \infty$ at $\alpha(y)$ and/or $\beta(y)$. Furthermore, choosing $c$ small we can make the projection of the support of $H$ lie in a region $\left[ \alpha(y) - k\Delta(y), \beta(y) + k\Delta(y)\right]$ where $\Delta(y) = \beta(y) -\alpha(y)$ and $k$ depends on $c$.  So the support of $H$ is contained in a region looking as in Figure \ref{fig:nbhdLproj} (where $\theta$ and $k$ can be made arbitrarily small).
\begin{figure}[h] 
  \centering
  \includegraphics[width=3.25in,height=2.5in,keepaspectratio]{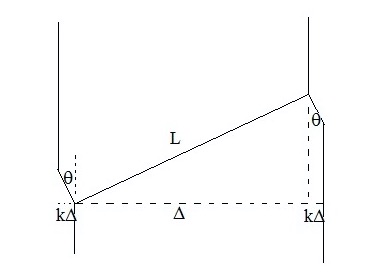}
  \caption{A region containing the support of $H$}
  \label{fig:nbhdLproj}
\end{figure}


\bigskip

{\bf Case III}. 
Now suppose a region $D$ as we have just described is surrounded on both sides by regions $D_1$ between [$\alpha_1(y), \alpha(y)\equiv \beta_1(y)$] and $D_2$ between 
[$\alpha_2(y)\equiv \beta(y), \beta_2(y)$]  on which there exist $r, r_1, r_2$ such that
\begin{eqnarray*}
\alpha(y) - \alpha_1(y) &\approx& t^{r_1},\\
\beta(y)-\alpha(y) &\approx& t^r,\\
\beta_2(y) -\beta(y) &\approx& t^{r_2},
\end{eqnarray*}
and such that $\Gamma f|D_1$ and $\Gamma f|D_2$ are FL.  See Figure \ref{fig:differentcurves}.
\begin{figure}[h] 
  \centering
  \includegraphics[width=4in,height=2.75in,keepaspectratio]{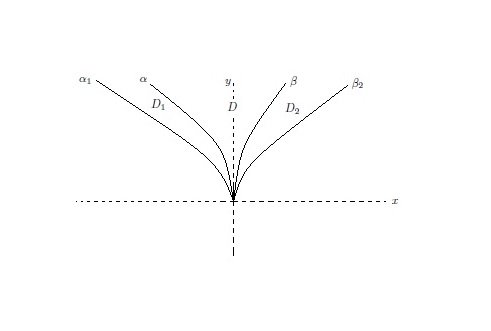}
  \caption{A region $D$ surrounded by regions $D_1$ and $D_2$}
  \label{fig:differentcurves}
\end{figure}

Assume that $r_1, r_2 \le r$ (so the width of $D$ is no greater than that $D_1$ or $D_2$).   Then there is a semialgebraic bilipschitz $H$ taking $V_{(D_1\cup D\cup D_2)} = \Gamma f|_{(D_1\cup D \cup D_2)}$ to its piecewise linearization over each $D_1, D_2, D$ with support contained over $(D_1\cup D \cup D_2)\times z\hbox{-axis}.$  But if 
$h(\Gamma f|_D)>w(\Gamma f|_{D_1\cup D \cup D_2})$,
then the linearization over $D$ approaches infinite slope as $y$ goes to 0.
\bigskip

{\bf Case IV}.  Assume everything is as in Case III, but in addition that 
$h(\Gamma f|_D)\le \min(w(\Gamma f|_{D_1}),w(\Gamma f|_{D_2}))$. Then divide
$D_1$ (resp. $D_2$) down the middle to get $D^L_1$ and $D^R_1$ (resp.
$D^L_2$ and $D^R_2$), and let $DD=D^R_1\cup D\cup D^L_2$.  
Then Case III can be applied 
to the triple $D^L_1$, $DD$, $D^R_2$ getting that $DD$ is semialgebraic
bilipschitz equivalent to its linearization.  Furthermore $w(DD)\ge h(DD)$,
so this linearization is the graph of a lipschitz function.  In addition,
the support of the equivalence map is compact and projects to 
$D_1\cup D\cup D_2$.  

Case IV is exactly what we need to establish part (a) of the Theorem.

If the hypotheses of (a) hold for all exceptional rays, then it is 
immediate that $V$ is semialgebraic bilipschitz equivalent to the union
of various linearizations, all of which are lipschitz, and so are 
semialgebraic bilipschitz equivalent to $C$, which proves (b).
\end{proof}

\begin{theorem}{\bf Inner Equivalence.}\label{thm: inner}

Assume that $V\subset \R^3$ is a semialgebraic surface that is a $C^1$-manifold everywhere except $\0\in V$.  Assume further that $V = \Gamma f$ is a graph of a semialgebraic function $f: U \rightarrow \R$, continuous at $\0$, $U$ a semialgebraic neighborhood of $\0$, and that the tangent cone $C\equiv CV$ is the $xy$-plane. Then $V$ is 
semialgebraically inner bilipschitz equivalent to $C$.
\end{theorem}

\begin{proof}
We will use the local bilipschitz classification given in 
\cite{B}.  There it is shown that two semialgebraic surfaces are semialgebraically
inner bilipschitz equivalent if they have the same (or combinatorially equivalent)
H\"older complexes.  A H\"older complex is the pair of a graph and a function associating
to each edge $e$ a rational number $\beta(e)\ge 1$.  The edges correspond to plane regions between two
arcs from 0 so that the width of the plane region is the reciprocal of the rational number
associated to the edge.  

$V$ is the consecutive union of $n$ pieces which are either FL, FI or FD, and each FI or FD piece is bounded by arcs tangent to an exceptional ray $\ell$. Each of these pieces gives one edge in a cyclic graph.  Each FL piece is 
semialgebraically lipschitz equivalent to the projection $U$ of that piece in the $xy$-plane,
and we let $\beta(e)=1/w(U)$.  Each FI or FD piece is semialgebraically lipschitz equivalent to the projection $U$ of that piece in the plane spanned by $\ell$ and the $z$-axis.  The
height of that piece is the width of $U$, and is the reciprocal of $\beta(e)>1$ for the 
corresponding edge $e$.  This gives a H\"older complex for $V$.

Using the same cyclic sequence of numbers $w_1,\dots,w_n$ (reciprocals of the $\beta(e_i)$'s),
we can subdivide the $xy$-plane into consecutive pieces of these widths.  So $V$ and $\R^2$ have combinatorially equivalent H\"older complexes, so are semialgebraically inner bilipschitz equivalent.
\end{proof}

\end{document}